\documentclass[12pt,oneside,english]{amsart}
\usepackage[T1]{fontenc}
\usepackage[latin9]{inputenc}
\usepackage{amsthm}
\usepackage{amsbsy}
\usepackage{amstext}
\usepackage{amssymb}
\usepackage{esint}

\makeatletter
\numberwithin{equation}{section}
\numberwithin{figure}{section}
\theoremstyle{plain}
\newtheorem{thm}{\protect\theoremname}
  \theoremstyle{remark}
  \newtheorem{acknowledgement}[thm]{\protect\acknowledgementname}
  \theoremstyle{remark}
  \newtheorem{rem}[thm]{\protect\remarkname}
  \theoremstyle{definition}
  \newtheorem{defn}[thm]{\protect\definitionname}
  \theoremstyle{plain}
  \newtheorem{assumption}[thm]{\protect\assumptionname}
  \theoremstyle{plain}
  \newtheorem{lem}[thm]{\protect\lemmaname}
  \theoremstyle{plain}
  \newtheorem{cor}[thm]{\protect\corollaryname}

\usepackage[a4paper]{geometry}
\usepackage{libertine} 

\makeatother

\usepackage{babel}
  \providecommand{\acknowledgementname}{Acknowledgement}
  \providecommand{\assumptionname}{Assumption}
  \providecommand{\corollaryname}{Corollary}
  \providecommand{\definitionname}{Definition}
  \providecommand{\lemmaname}{Lemma}
  \providecommand{\remarkname}{Remark}
\providecommand{\theoremname}{Theorem}

\begin{document}

\title[A Levy-area between brownian motion and rough paths with applications]{A levy-area between brownian motion and rough paths with applications
to robust non-linear filtering and rpdes }

\author{Joscha Diehl}

\address{TU Berlin, Institut für Mathematik, Straße des 17.~Juni 136, 10623
Berlin}

\email{diehl@math.tu-berlin.de}

\author{Harald Oberhauser}

\address{University of Oxford, Oxford--Man Institute, Eagle House, Walton
Well Road}

\email{harald.oberhauser@oxford-man.ox.ac.uk}

\author{sebastian riedel}

\address{TU Berlin, Institut für Mathematik, Straße des 17.~Juni 136, 10623
Berlin}

\email{riedel@math.tu-berlin.de}


\begin{abstract}
We give meaning and study the regularity of differential equations
with a rough path term and a Brownian noise term, that is we are interested
in equations of the type 
\begin{align*}
S_{t}^{\boldsymbol{\eta}} & =S_{0}+\int_{0}^{t}a\left(S_{r}^{\boldsymbol{\eta}}\right)dr+\int_{0}^{t}b\left(S_{r}^{\boldsymbol{\eta}}\right)\circ dB_{r}+\int_{0}^{t}c\left(S_{r}^{\boldsymbol{\eta}}\right)d\boldsymbol{\eta}_{r}
\end{align*}
where $\boldsymbol{\eta}$ is a deterministic geometric, step-$2$
rough path and $B$ is a multi-dimensional Brownian motion. En passant,
we give a short and direct argument that implies integrability estimates
for rough differential equations with Gaussian driving signals which
is of independent interest.
\end{abstract}
\maketitle

\section{Introduction}

The contribution of this article is twofold: firstly, we give meaning
to differential equations of the type 
\begin{equation}
S_{t}^{\boldsymbol{\eta}}=S_{0}+\int_{0}^{t}a\left(S_{r}^{\boldsymbol{\eta}}\right)dr+\int_{0}^{t}b\left(S_{r}^{\boldsymbol{\eta}}\right)\circ dB_{r}+\int_{0}^{t}c\left(S_{r}^{\boldsymbol{\eta}}\right)d\boldsymbol{\eta}_{r},\label{eq:rsde-1}
\end{equation}
that is, for a deterministic, step $2$-rough path $\boldsymbol{\eta}$
we are looking for a stochastic process $S^{\boldsymbol{\eta}}$ that
is adapted to $\sigma\left(B\right)$ and study the regularity of
the map $\boldsymbol{\eta}\mapsto S^{\boldsymbol{\eta}}$. Secondly,
we take this as an opportunity to revisit the integrability estimates
of solutions of rough differential equations driven by Gaussian processes.

If either $b\equiv0$ or $c\equiv0$ then rough path theory \cite{lyons-98,lyons-qian-02,lyons-04,MR2091358,friz-victoir-book}
or standard It\={o}-calculus allow (under appropriate regularity assumptions
on the vector fields $\left(a,b,c\right)$) to give meaning to (\ref{eq:rsde-1}).
However, in the generic case when the vector fields $b$ and $c$
have a non-trivial Lie bracket, any notion of a solution (that is
consistent with an It\={o}--Stratonovich calculus) must take into
account the area swept out between the trajectories of $B$ and $\boldsymbol{\eta}$.
A natural approach is to identify $S^{\boldsymbol{\eta}}$ as the
RDE solution of 
\begin{equation}
S_{t}=S_{0}+\int_{0}^{t}\left(a,b,c\right)\left(S_{r}\right)d\left(r,\boldsymbol{\Lambda}_{r}\right)\label{eq:rsde}
\end{equation}
where $\boldsymbol{\Lambda}$ is a joint, step-$2$ rough path lift
between the enhanced Brownian motion $\boldsymbol{B}=\left(1+B+\int B\otimes\circ dB\right)$
and $\mathcal{\boldsymbol{\eta}}$, and $\left(r,\boldsymbol{\Lambda}\right)$
is the joint rough path between the random rough path $\boldsymbol{\Lambda}$
and the bounded variation path $r\mapsto r$. While the existence
of a joint lift between a continuous bounded variation path and any
rough path is trivial (via integration by parts), the existence of
a joint lift between two given step-$2$ rough paths is more subtle
and in general not possible. More precisely, let $\alpha\in\left(\frac{1}{3},\frac{1}{2}\right)$
and denote with $\mathcal{C}^{0,\alpha}\left(\mathbb{R}^{d}\right)$
the space of geometric, step-$2$, $\alpha$-Hölder rough paths over
$\mathbb{R}^{d}$ (we often only write $\mathcal{C}^{0,\alpha}$ and
$d$ is chosen according to context). Fix two geometric, step-2 rough
paths $\boldsymbol{\eta}=\left(1+\boldsymbol{\eta}^{1}+\boldsymbol{\eta}^{2}\right)\in\mathcal{C}^{0,\alpha}\left(\mathbb{R}^{d}\right)$,
$\boldsymbol{b}=\left(1+\boldsymbol{b}^{1}+\boldsymbol{b}^{2}\right)\in\mathcal{C}^{0,\alpha}\left(\mathbb{R}^{e}\right)$.
In general, one cannot hope to find a joint rough path lift, i.e.~a
geometric rough path $\boldsymbol{\lambda}=\left(1+\boldsymbol{\lambda}^{1}+\boldsymbol{\lambda}^{2}\right)\in\mathcal{C}^{0,\alpha}\left(\mathbb{R}^{d+e}\right)$
such that (formally) 
\[
\boldsymbol{\lambda}^{1}=\left(\boldsymbol{\eta}^{1},\boldsymbol{b}^{1}\right)\text{ and }\boldsymbol{\lambda}^{2}=\left(\begin{array}{cc}
\boldsymbol{\eta}^{2} & \int\eta\otimes db\\
\int b\otimes d\eta & \boldsymbol{b}^{2}
\end{array}\right)
\]
since the entries on the cross--diagonal of $\boldsymbol{\lambda}^{2}$
are not well--defined.  (What is guaranteed by the extension theorem
in \cite{Lyons2007835} is that there exists a weak geometric rough
path $\overline{\boldsymbol{\lambda}}$ such that \textbf{$\overline{\boldsymbol{\lambda}}^{1}=\left(\boldsymbol{\eta}^{1},\boldsymbol{b}^{1}\right)$},
however this $\mathcal{\boldsymbol{\overline{\lambda}}}$ is highly
non-unique and no consistency with $\boldsymbol{\eta}$ or $\boldsymbol{b}$
on the second level is guaranteed). 

In Section \ref{sec:the-joint-lift} we show that in the case when
the deterministic rough path $\boldsymbol{b}$ is replaced by enhanced
Brownian motion $\boldsymbol{B}$, there does indeed exists a stochastic
process $\boldsymbol{\Lambda}$ which merits in a certain sense to
be called the ``canonical joint lift'' of $\boldsymbol{\eta}$ and
$B$. In Section \ref{sec:differential-equations-with} we use this
lift $\boldsymbol{\Lambda}$ to give meaning to differential equations
(\ref{eq:rsde-1}) resp.~(\ref{eq:rsde}) and establish local Lipschitzness
of the solution map $\boldsymbol{\eta}\mapsto S^{\boldsymbol{\eta}}$
from the space of geometric rough paths equipped with Hölder metric
into the space of stochastic processes adapted to the Brownian filtration
equipped with the topology of uniform convergence in $L^{q}\left(\Omega\right)$-norm.
This is exactly the type of robustness we are interested in and finally
allows us to turn to our initial motivation: differential equations
of the form (\ref{eq:rsde-1}) naturally arise in certain robustness
problems and were previously treated with a flow decomposition which
ultimately leads to stronger regularity assumptions on the vector
fields. In Section \ref{sec:Applications} we give two such applications.
One revisits Clark's robustness problem in nonlinear filtering and
provides an alternative to the recent approach via flow decomposition
carried out in \cite{CrisanDiehlFrizOberhauser}, the other one is
a Feynman--Kac representation of solutions of PDEs with linear rough
path noise.

Our application to stochastic filtering demands exponential integrability
of differential equations driven by $\boldsymbol{\Lambda}$. We take
this as an opportunity to revisit existing results on integrability
estimates for rough paths and rough differential equations in a general
setup (which then even implies Gaussian integrability for differential
equations driven by $\boldsymbol{\Lambda}$ that are uniform in $\boldsymbol{\eta}$).
In Section \ref{sec:integrability-and-tail} we give a surprisingly
short proof of the integrability properties of RDEs driven by Gaussian
rough paths by revisiting and combining the key insights from \cite{MR2661566}
and \cite{bibCassLyonsLitterer} in a direct and tractable way which
we think is of independent interest.
\begin{acknowledgement}
JD, HO and SR were supported by the European Research Council under
the European Union's Seventh Framework Programme ERC grant agreement
nr.~258237. JD was also supported by DFG Grant SPP-1324. HO was also
supported by the European Research Council under the European Union's
Seventh Framework Programme ERC grant agreement nr.~291244 and by
the Oxford--Man Institute.
\begin{acknowledgement}
The authors would like to thank Thomas Cass, Dan Crisan, Peter Friz
and Terry Lyons for helpful conversations. 
\end{acknowledgement}
\end{acknowledgement}

\section{\label{sec:the-joint-lift}The joint lift}

As usual we denote with $Lip^{\gamma}$ the set of $\gamma$-Lipschitz
functions $a:\mathbb{R}^{d_{1}}\to\mathbb{R}^{d_{2}}$ in the sense
of E.~Stein%
\footnote{That is bounded $k$-th derivative for $k=0,\ldots,\left\lfloor \gamma\right\rfloor $
and $\left(\gamma-\left\lfloor \gamma\right\rfloor \right)$-Hölder
continuous $\left\lfloor \gamma\right\rfloor $-th derivative, where
$\left\lfloor \gamma\right\rfloor $ is the largest integer strictly
smaller then $\gamma$. %
} where $d_{1}$ and $d_{2}$ are chosen according to the context.
$G_{d}^{2}\cong\mathbb{R}^{d}\oplus so\left(d\right)$ is the free
nilpotent group%
\footnote{This is the correct state space for a geometric $1/p$-Hölder rough
path; the space of such paths subject to $1/p$-Hölder regularity
(in rough path sense) yields a complete metric space under $1/p$-Hölder
rough path metric. Technical details of geometric rough path spaces
can be found e.g. in Section 9 of \cite{friz-victoir-book}.%
} of step $2$ over $\mathbb{R}^{d}$. We equip the space of geometric
rough paths with the non-homogeneous metric $\rho_{\alpha-H\ddot{o}l}$
which makes it a Polish space, denoted $\mathcal{C}^{0,\alpha}$,
and denote the associated non-homogeneous norm%
\footnote{We denote norms on linear spaces with $\left|.\right|$ and ``norms''
on non-linear spaces (like $G_{d}^{2}$ or $\mathcal{C}_{d}^{0,\alpha}$)
with $\left\Vert .\right\Vert $.%
} on this non-linear space with $\left\Vert .\right\Vert _{\alpha-H\ddot{o}l}$,
(similarly we denote the non-separable space of weak geometric rough
paths with $\mathcal{C}^{\alpha}\left(\mathbb{R}^{d}\right)$, cf.~\cite[Chapter 9.2]{friz-victoir-book}).
\begin{thm}
\label{thm:joint lift}Let $\alpha\in\left(\frac{1}{3},\frac{1}{2}\right)$,
$\boldsymbol{\eta}\in\mathcal{C}^{0,\alpha}\left(\mathbb{R}^{d}\right)$
and $B=\left(B^{i}\right)_{i=1}^{e}$ be an $e$-dimensional Brownian
motion carried on a probability space $\left(\Omega,\mathcal{F},\mathcal{F}_{t},\mathbb{P}\right)$
satisfying the usual conditions. Then for every $\alpha'<\alpha$
there exists a $\mathcal{C}^{0,\alpha'}\left(\mathbb{R}^{d+e}\right)$
-valued random variable $\boldsymbol{\Lambda}=\boldsymbol{\Lambda}^{\boldsymbol{\eta}}$
on $\left(\Omega,\mathcal{F},\mathcal{F}_{t},\mathbb{P}\right)$ which
fulfills $\mathbb{P}$-a.s.~that for every $t\geq0$,
\begin{eqnarray}
\boldsymbol{\Lambda}_{t}^{1;i} & = & \begin{cases}
\boldsymbol{\eta}_{t}^{1;i} & \text{, if }i\in\left\{ 1,\ldots,d\right\} \\
B_{t}^{i-d} & \text{, if }i\in\left\{ d+1,\ldots,d+e\right\} 
\end{cases}\label{eq:consistent}\\
\boldsymbol{\Lambda}_{t}^{2;i,j} & = & \begin{cases}
\boldsymbol{\eta}_{t}^{2;i,j} & \text{, if }i,j\in\left\{ 1,\ldots,d\right\} \\
\int_{0}^{t}B_{r}^{i-d}\circ dB_{r}^{j-d} & \text{, if }i,j\in\left\{ d+1,\ldots,d+e\right\} .
\end{cases}\nonumber 
\end{eqnarray}
Moreover,\renewcommand*\theenumi{\roman{enumi}}
\begin{enumerate}
\item \textup{$\boldsymbol{\Lambda}^{\boldsymbol{\eta}}$ }has Gaussian
tails, locally uniform in $\boldsymbol{\eta}$: $\forall r>0$ $\exists$$\delta=\delta\left(\alpha',\alpha,T,r\right)>0$
such that \textup{
\begin{equation}
\sup_{\left\Vert \boldsymbol{\eta}\right\Vert {}_{\alpha-H\ddot{o}l}\le r}\mathbb{E}\exp\left(\delta\left\Vert \boldsymbol{\Lambda}^{\boldsymbol{\eta}}\right\Vert {}_{\alpha'-H\ddot{o}l}^{2}\right)<\infty.\label{eq:Gauss tail lambda}
\end{equation}
}
\item $\boldsymbol{\eta}\mapsto\boldsymbol{\Lambda}^{\boldsymbol{\eta}}$
is locally Lipschitz in $L^{q}$: $\forall r>0$, there exists a constant
$c_{Lip}=c_{Lip}\left(r,q,\alpha,\alpha'\right)$ such that for all
$\boldsymbol{\eta},\bar{\boldsymbol{\eta}}\in\mathcal{C}^{0,\alpha}\left(\mathbb{R}^{d}\right)$
with $\left\Vert \boldsymbol{\eta}\right\Vert {}_{\alpha-H\ddot{o}l}$,$\left\Vert \bar{\boldsymbol{\eta}}\right\Vert {}_{\alpha-H\ddot{o}l}\le r$
\[
\left|\rho_{\alpha'-H\ddot{o}l}\left(\boldsymbol{\Lambda}^{\boldsymbol{\eta}},\boldsymbol{\Lambda}^{\overline{\boldsymbol{\eta}}}\right)\right|_{L^{q}\left(\Omega;\mathbb{R}\right)}\le c_{Lip}\rho_{\alpha-H\ddot{o}l}\left(\boldsymbol{\eta},\bar{\boldsymbol{\eta}}\right).
\]

\item $\boldsymbol{\Lambda}$ is consistent with the Stratonovich lift for
semimartingales: let $N$ be a multidimensional continuous semimartingale
carried on another probability space $\left(\overline{\Omega},\overline{\mathcal{F}},\overline{\mathcal{F}}_{t},\overline{\mathbb{P}}\right)$
and consider the product space with $\left(\Omega,\mathcal{F},\mathcal{F}_{t},\mathbb{P}\right)$
equipped with $\overline{\mathbb{P}}\otimes\mathbb{P}$. Denote with
$\boldsymbol{N}$ resp.~$\left(\boldsymbol{N},\boldsymbol{B}\right)$
the Stratonovich lift of the semimartingales $N$ resp.~$\left(N,B\right)$.
Then for $\overline{\mathbb{P}}$-a.e.~$\overline{\omega}\in\overline{\Omega}$
we have 
\[
\mathbb{P}\left[\omega:\boldsymbol{\Lambda}^{\boldsymbol{N}\left(\overline{\omega}\right)}\left(\omega\right)=\left(\boldsymbol{N},\boldsymbol{B}\right)\left(\overline{\omega}\otimes\omega\right)\right]=1.
\]

\end{enumerate}
\end{thm}
\begin{proof}
Define 
\begin{eqnarray}
\boldsymbol{\Lambda}_{t}^{1;i} & := & \begin{cases}
\boldsymbol{\eta}^{1;i} & \text{, if }i\in\left\{ 1,\ldots,d\right\} \\
B^{i-d} & \text{, if }i\in\left\{ d+1,\ldots,d+e\right\} 
\end{cases},\nonumber \\
\boldsymbol{\Lambda}_{t}^{2;i,j} & := & \begin{cases}
\boldsymbol{\eta}^{2;i,j} & \text{, if }i,j\in\left\{ 1,\ldots,d\right\} \\
\int_{0}^{t}B_{r}^{i-d}\circ dB_{r}^{j-d} & \text{, if }i,j\in\left\{ d+1,\ldots,d+e\right\} \\
\int_{0}^{t}\eta_{u}^{i}dB_{u}^{j-d} & \text{, if }i\in\left\{ 1,\ldots,d\right\} ,j\in\left\{ 1+d,\ldots,d+e\right\} \\
\eta_{t}^{j}B_{t}^{i-d}-\int_{0}^{t}\eta_{u}^{j}dB_{u}^{i-d} & \text{, if }i\in\left\{ d+1,\ldots,d+e\right\} ,j\in\left\{ 1,\ldots,d\right\} .
\end{cases}\label{eq:lambda2}
\end{eqnarray}
Then $\boldsymbol{\Lambda}_{t}=\left(1+\boldsymbol{\Lambda}_{t}^{1}+\boldsymbol{\Lambda}_{t}^{2}\right)\in1+\mathbb{R}^{d+e}+\left(\mathbb{R}^{d+e}\right)^{\otimes2}$
and a direct calculation shows that $\boldsymbol{\Lambda}_{s,t}:=\boldsymbol{\Lambda}_{s}^{-1}\otimes\boldsymbol{\Lambda}_{t}=\exp\left[\left(\eta_{s,t},B_{s,t}\right)+A_{s,t}\right]$
with%
\footnote{We could define $\boldsymbol{\Lambda}$ directly via (\ref{eq:area})
but the above way might be a bit more intuitive.%
} 
\begin{equation}
so\left(d\right)\ni A_{s,t}^{i,j}=\left\{ \begin{array}{ll}
\frac{1}{2}\left(\int_{s}^{t}\eta_{s,u}^{i}d\eta_{u}^{j}-\int_{s}^{t}\eta_{s,u}^{j}d\eta_{u}^{i}\right) & \text{ , if }i,j\in\left\{ 1,\ldots,d\right\} \\
\frac{1}{2}\left(\int_{s}^{t}B_{s,u}^{i-d}dB_{u}^{j-d}-\int_{s}^{t}B_{s,u}^{j-d}dB_{u}^{i-d}\right) & \text{ , if }i,j\in\left\{ d+1,\ldots,d+e\right\} \\
\left(\int_{s}^{t}\eta_{s,u}^{i}dB_{u}^{j-d}-\frac{1}{2}\eta_{s,t}^{i}B_{s,t}^{j-d}\right) & \text{ , if }i\in\left\{ 1,\ldots,d\right\} ,j\in\left\{ 1+d,\ldots,d+e\right\} \\
\left(-\int_{0}^{t}\eta_{u}^{j}dB_{u}^{i-d}+\frac{1}{2}\eta_{s,t}^{j}B_{s,t}^{i-d}\right) & \text{ , if }i\in\left\{ d+1,\ldots,d+e\right\} ,j\in\left\{ 1,\ldots,d\right\} .
\end{array}\right.\label{eq:area}
\end{equation}
That is, (after throwing away a null-set depending on $B$ and $\boldsymbol{\eta}$)
we have shown that $t\mapsto\boldsymbol{\Lambda}_{t}$ is a continuous
path that takes values in $G^{2}\left(\mathbb{R}^{d+e}\right)$. It
remains to demonstrate that $\left\Vert \boldsymbol{\Lambda}\right\Vert _{\alpha'-H\ddot{o}l}<\infty$
for any $\alpha'<\alpha$ which is then enough to conclude that $\boldsymbol{\Lambda}\in\mathcal{C}^{0,\alpha'}\left(\mathbb{R}^{d+e}\right)$
$\mathbb{P}$-a.s.~for $\alpha'<\alpha$ due to the embedding of
weak geometric rough paths into geometric rough paths ($\mathcal{C}^{\beta}\subset\mathcal{C}^{0,\beta'}$
for $\beta>\beta'$ follows from \cite[Theorem 19]{friz-victoir-04-Note}).
We show this Hölder regularity by proving a stronger statement, namely
that $\left\Vert \boldsymbol{\Lambda}\right\Vert {}_{\alpha'}$ has
Gaussian tails. It is clear that the first level of $\boldsymbol{\Lambda}$
is $\alpha$-Hölder; to deal with the second level note that $\left|A_{st}\right|\sim\sum_{ij}\left|A_{st}^{ij}\right|$
and that $2\alpha-$Hölder regularity already holds for the first
two cases, that is $i,j\in\left\{ 1,\ldots,d\right\} $ and $i,j\in\left\{ d+1,\ldots,d+e\right\} $
(in fact even a Gauss tail via the Fernique estimate for rough path
norms \cite{MR2661566}). Now for the remaining case we have from
above definition of $A^{i,j}$ that 
\begin{eqnarray*}
\left|A_{s,t}^{i,j}\right| & \leq & \left|\int_{s}^{t}\eta_{s,u}^{i}dB_{u}^{j}\right|+\frac{1}{2}\left|\eta_{s,t}^{i}B_{s,t}^{j}\right|
\end{eqnarray*}
and since $\left|\eta_{s,t}^{i}B_{s,t}^{j}\right|\leq\left|\eta\right|_{\alpha-H\ddot{o}l}\left|B\right|_{\alpha-H\ddot{o}l}\left|t-s\right|^{2\alpha}$
it just remains to treat $\left|\int_{s}^{t}\eta_{s,u}^{i}dB_{u}^{j}\right|$.
But since for each $s<t$, $\int_{s}^{t}\eta_{s,u}^{i}dB_{u}^{j}$
has the same distribution as $\sqrt{\int_{s}^{t}\left(\eta_{s,u}^{i}\right)^{2}du}.Z$
for some fixed $Z\sim\mathcal{N}\left(0,1\right)$, we also have
\[
\exp\left[\kappa\left(\frac{\left|\int_{s}^{t}\left(\eta_{s,u}^{i}\right)dB_{u}^{j}\right|}{\left(t-s\right)^{2\alpha}}\right)^{2}\right]=^{Law}\exp\left[\kappa Z^{2}\left(\frac{\sqrt{\left|\int_{s}^{t}\left(\eta_{s,u}^{i}\right)^{2}du\right|}}{\left(t-s\right)^{2\alpha}}\right)^{2}\right]
\]
and by the elementary estimate $\left|\int_{s}^{t}\left(\eta_{s,u}^{i}\right)^{2}du\right|\leq\left\Vert \eta\right\Vert _{\alpha-H\ddot{o}l}^{2}\left(t-s\right)^{2\alpha+1}$
we can conclude by taking $\sup_{s<t}\mathbb{E}$ in above expression
and using the Gaussian integrability for $Z$, i.e.~there exists
a $\kappa>0$ such that 
\[
\sup_{s<t}\mathbb{E}\exp\left[\kappa Z^{2}\left|\eta\right|_{\alpha-H\ddot{o}l}^{2}\left(t-s\right)^{1-2\alpha}\right]<\infty.
\]
By \cite[Theorem A.19]{friz-victoir-book} this yields the desired
$2\alpha'$-Hölder regularity of $\left|\int_{s}^{t}\eta_{s,u}^{i}dB_{u}^{j}\right|$
for any $\alpha'<\alpha$. Putting everything together, we have shown
that $\left\Vert \boldsymbol{\Lambda}\right\Vert _{\alpha'-H\ddot{o}l}<\infty$,
$\mathbb{P}$-a.s. In fact we even have shown 
\[
\sup_{\left\Vert \boldsymbol{\eta}\right\Vert {}_{\alpha-H\ddot{o}l}\le r}\mathbb{E}\exp\left(\delta\left\Vert \boldsymbol{\Lambda}^{\boldsymbol{\eta}}\right\Vert {}_{\alpha'-H\ddot{o}l}^{2}\right)<\infty,\,\,\,\forall\alpha'<\alpha.
\]
 It remains to show the claimed Lipschitz continuity of the map $\boldsymbol{\eta}\mapsto\boldsymbol{\Lambda}^{\boldsymbol{\eta}}$.
Therefore let $q\ge q_{0}\left(\alpha,\alpha'\right)$, as given in
\cite[Theorem A.13]{friz-victoir-book}, take $\boldsymbol{\eta},\bar{\boldsymbol{\eta}}\in\mathcal{C}^{0,\alpha}$
with $\left\Vert \boldsymbol{\eta}\right\Vert {}_{\alpha-H\ddot{o}l},\left\Vert \bar{\boldsymbol{\eta}}\right\Vert {}_{\alpha-H\ddot{o}l}\le r$
and denote the corresponding lifts $\boldsymbol{\Lambda}^{\boldsymbol{\eta}},\boldsymbol{\Lambda}^{\overline{\boldsymbol{\eta}}}$.
Set $\varepsilon:=\rho_{\alpha'-H\ddot{o}l}\left(\boldsymbol{\Lambda}^{\boldsymbol{\eta}},\boldsymbol{\Lambda}^{\overline{\boldsymbol{\eta}}}\right)$.
By (\ref{eq:Gauss tail lambda}) there exists a constant $c_{1}=c_{1}\left(q,r\right)$
such that (we denote the Carnot--Caratheodory metric on $G_{d+e}^{2}$
with $d_{CC}$) 
\[
\left|d_{CC}\left(\boldsymbol{\Lambda}_{s}^{\boldsymbol{\eta}},\boldsymbol{\Lambda}_{t}^{\boldsymbol{\eta}}\right)\right|_{L^{q}\left(\Omega;\mathbb{R}\right)}^{q},\left|d_{CC}\left(\boldsymbol{\Lambda}_{s}^{\overline{\boldsymbol{\eta}}},\boldsymbol{\Lambda}_{t}^{\overline{\boldsymbol{\eta}}}\right)\right|_{L^{q}\left(\Omega;\mathbb{R}\right)}^{q}\le c_{1}\left|t-s\right|{}^{\alpha q}.
\]
Moreover 
\begin{align*}
\left|\pi_{1}\left(\boldsymbol{\Lambda}_{s,t}^{\boldsymbol{\eta}}-\boldsymbol{\Lambda}_{s,t}^{\overline{\boldsymbol{\eta}}}\right)\right|_{L^{q}\left(\Omega;\mathbb{R}\right)}^{q} & =\left|\eta_{s,t}-\bar{\eta}_{s,t}\right|{}^{q}\\
 & \le\varepsilon^{q}\left|t-s\right|{}^{\alpha q},
\end{align*}
and (again the constants $c$ may only depend on $r$ and $q$) 
\begin{align*}
\left|\pi_{2}\left(\boldsymbol{\Lambda}_{s,t}^{\boldsymbol{\eta}}-\boldsymbol{\Lambda}_{s,t}^{\overline{\boldsymbol{\eta}}}\right)\right|_{L^{q/2}\left(\Omega;\mathbb{R}\right)}^{2/q} & =\left|\frac{1}{2}\pi_{1}\left(\boldsymbol{\Lambda}_{s,t}^{\boldsymbol{\eta}}-\boldsymbol{\Lambda}_{s,t}^{\overline{\boldsymbol{\eta}}}\right)\otimes\pi_{1}\left(\boldsymbol{\Lambda}_{s,t}^{\boldsymbol{\eta}}-\boldsymbol{\Lambda}_{s,t}^{\overline{\boldsymbol{\eta}}}\right)+A_{s,t}-\bar{A}_{s,t}\right|_{L^{q/2}\left(\Omega;\mathbb{R}\right)}^{2/q}\\
 & \le c\left(\varepsilon^{q/2}\left|t-s\right|{}^{\alpha q}+\left|A_{s,t}-\bar{A}_{s,t}\right|_{L^{q/2}\left(\Omega;\mathbb{R}\right)}^{2/q}\right)
\end{align*}
Also, 
\begin{eqnarray*}
 &  & \left|\int_{s}^{t}\eta_{s,u}^{i}d\eta_{u}^{j}-\int_{s}^{t}\eta_{s,u}^{j}d\eta_{u}^{i}-\left(\int_{s}^{t}\bar{\eta}_{s,u}^{i}d\bar{\eta}_{u}^{j}-\int_{s}^{t}\bar{\eta}_{s,u}^{j}d\bar{\eta}_{u}^{i}\right)\right|_{L^{q/2}\left(\Omega;\mathbb{R}\right)}^{2/q}\\
 & = & \left|\int_{s}^{t}\eta_{s,u}^{i}d\eta_{u}^{j}-\int_{s}^{t}\eta_{s,u}^{j}d\eta_{u}^{i}-\left(\int_{s}^{t}\bar{\eta}_{s,u}^{i}d\bar{\eta}_{u}^{j}-\int_{s}^{t}\bar{\eta}_{s,u}^{j}d\bar{\eta}_{u}^{i}\right)\right|^{2/q}\\
 & \le & \varepsilon^{q/2}\left|t-s\right|{}^{\alpha q},
\end{eqnarray*}
and finally 
\begin{eqnarray*}
 &  & \left|\int_{s}^{t}\eta_{s,u}^{i}dB_{u}^{j-d}-\frac{1}{2}\eta_{s,t}^{i}B_{s,t}^{j-d}-\left(\int_{s}^{t}\bar{\eta}_{s,u}^{i}dB_{u}^{j-d}-\frac{1}{2}\bar{\eta}_{s,t}^{i}B_{s,t}^{j-d}\right)\right|_{L^{q/2}\left(\Omega;\mathbb{R}\right)}^{2/q}\\
 & \leq & c\left|\int_{s}^{t}\eta_{s,u}^{i}-\bar{\eta}_{s,u}^{i}dB_{u}^{j-d}\right|_{L^{q/2}\left(\Omega;\mathbb{R}\right)}^{2/q}+c\left|\eta_{s,t}^{i}B_{s,t}^{j-d}-\bar{\eta}_{s,t}^{i}B_{s,t}^{j-d}\right|_{L^{q/2}\left(\Omega;\mathbb{R}\right)}^{2/q}\\
 & \le & c\left|\int_{s}^{t}\left|\eta_{s,u}^{i}-\bar{\eta}_{s,u}^{i}\right|^{2}du\right|^{q/4}+c\left|\eta_{s,t}^{i}-\eta_{s,t}^{i}\right|{}^{q/2}\left|B_{s,t}^{j-d}\right|_{L^{q/2}\left(\Omega;\mathbb{R}\right)}^{2/q}\\
 & \le & c\varepsilon^{q/2}\left|t-s\right|{}^{\alpha q/2+q/4}+c\varepsilon^{q/2}\left|t-s\right|{}^{\alpha q/2}\left|t-s\right|{}^{q/4}\\
 & \le & c\varepsilon^{q/2}\left|t-s\right|{}^{\alpha q}.
\end{eqnarray*}
Hence, 
\[
\left|\pi_{2}\left(\boldsymbol{\Lambda}_{s,t}^{\boldsymbol{\eta}}-\boldsymbol{\Lambda}_{s,t}^{\overline{\boldsymbol{\eta}}}\right)\right|_{L^{q/2}\left(\Omega;\mathbb{R}\right)}^{2/q}\le c_{2}\varepsilon^{q/2}\left|t-s\right|{}^{\alpha q}
\]
and applied with $m=m\left(r,q\right):=\max\left\{ 1,c_{1}^{1/q},c_{2}^{1/(2q)}\right\} $
we have $\forall q\ge1$ that 
\begin{eqnarray*}
\left|d_{CC}\left(\boldsymbol{\Lambda}_{s}^{\boldsymbol{\eta}},\boldsymbol{\Lambda}_{t}^{\overline{\boldsymbol{\eta}}}\right)\right|_{L^{q}\left(\Omega;\mathbb{R}\right)} & \le & m\left|t-s\right|{}^{\alpha},\\
\left|\pi_{1}\left(\boldsymbol{\Lambda}_{s,t}^{\boldsymbol{\eta}}-\boldsymbol{\Lambda}_{s,t}^{\overline{\boldsymbol{\eta}}}\right)\right|_{L^{q}\left(\Omega;\mathbb{R}\right)} & \le & \varepsilon m\left|t-s\right|{}^{\alpha},\\
\left|\pi_{2}\left(\boldsymbol{\Lambda}_{s,t}^{\boldsymbol{\eta}}-\boldsymbol{\Lambda}_{s,t}^{\overline{\boldsymbol{\eta}}}\right)\right|_{L^{q/2}\left(\Omega;\mathbb{R}\right)} & \le & \varepsilon m^{2}\left|t-s\right|{}^{2\alpha}.
\end{eqnarray*}
By \cite[Theorem A.13 (i)]{friz-victoir-book} there exists a $q$
large enough and a constant $k=k\left(\alpha,\alpha',T,q\right)$
such that 
\[
\left|\sup_{s<t}\frac{\left|\pi_{1}\left(\boldsymbol{\Lambda}_{s,t}^{\boldsymbol{\eta}}-\boldsymbol{\Lambda}_{s,t}^{\overline{\boldsymbol{\eta}}}\right)\right|}{\left|t-s\right|{}^{\alpha'}}\right|_{L^{q}\left(\Omega;\mathbb{R}\right)}\le\varepsilon km\text{ and }\left|\sup_{s<t}\frac{\left|\pi_{2}\left(\boldsymbol{\Lambda}_{s,t}^{\boldsymbol{\eta}}-\boldsymbol{\Lambda}_{s,t}^{\overline{\boldsymbol{\eta}}}\right)\right|}{\left|t-s\right|{}^{\alpha'}}\right|_{L^{q/2}\left(\Omega;\mathbb{R}\right)}\le\varepsilon\left(km\right){}^{2}
\]
Using this with $q$ and $2q$ we get from the definition of $\rho_{\alpha'-H\ddot{o}l}$
that 
\[
\left|\rho_{\alpha'-H\ddot{o}l}\left(\boldsymbol{\Lambda}^{\boldsymbol{\eta}},\boldsymbol{\Lambda}^{\overline{\boldsymbol{\eta}}}\right)\right|_{L^{q}\left(\Omega;\mathbb{R}\right)}\le\left|\sup_{s\not=t}\frac{\left|\pi_{1}\left(\boldsymbol{\Lambda}_{s,t}^{\boldsymbol{\eta}}-\boldsymbol{\Lambda}_{s,t}^{\overline{\boldsymbol{\eta}}}\right)\right|}{\left|t-s\right|{}^{\alpha'}}\right|_{L^{q}\left(\Omega;\mathbb{R}\right)}+\left|\sup_{s\not=t}\frac{\left|\pi_{2}\left(\boldsymbol{\Lambda}_{s,t}^{\boldsymbol{\eta}}-\boldsymbol{\Lambda}_{s,t}^{\overline{\boldsymbol{\eta}}}\right)\right|}{\left|t-s\right|{}^{2\alpha'}}\right|_{L^{q}\left(\Omega;\mathbb{R}\right)}\le c_{Lip}\varepsilon
\]
In the above argument we assumed that $q$ is large enough, but since
$L^{p}$ is Lipschitz continuously embedded in $L^{q}$ for $p>q$,
the result follows for all $q$. 

Above arguments imply (i) and (ii). We now establish point (iii).
Denote with $\left(\boldsymbol{N},\boldsymbol{B}\right)$ the usual
Stratonovich lift of the $\left(d+e\right)$-dimensional, continuous
semimartingale $\left(N,B\right)$ carried on the probability space
\[
\left(\hat{\Omega},\hat{\mathcal{F}},\hat{\mathcal{F}}_{t},\hat{\mathbb{P}}\right)=\left(\overline{\Omega}\times\Omega,\overline{\mathcal{F}}_{t}\otimes\mathcal{F}_{t},\overline{\mathcal{F}}\otimes\mathcal{F},\overline{\mathbb{P}}\otimes\mathbb{P}\right).
\]
We need to compare this lift for $\overline{\mathbb{P}}-a.e.$~$\overline{\omega}$
with the process $\boldsymbol{\Lambda}^{\boldsymbol{N}\left(\overline{\omega}\right)}$
defined on $\left(\Omega,\mathcal{F}_{t},\mathcal{F},\mathbb{P}\right)$.
Note that we cannot use that $\hat{\omega}=\left(\overline{\omega},\omega\right)\mapsto\boldsymbol{\Lambda}^{\boldsymbol{N}\left(\overline{\omega}\right)}\left(\omega\right)$
is a random variable on $\hat{\Omega}$, since the above argument
does not imply $\hat{\mathcal{F}}$-measurability (i.e.~ joint measurability
in $\left(\overline{\omega},\omega\right)$). However, for the components
of the first level, this is easily seen to be true: we immediately
get by the construction of $\boldsymbol{\Lambda}^{\boldsymbol{N}}$
that $\left(\overline{\omega},\omega\right)\mapsto\pi_{1}\left(\left(\boldsymbol{\Lambda}^{\boldsymbol{N}\left(\overline{\omega}\right)}\right)\left(\omega\right)\right)$
is $\hat{\mathcal{F}}$-measurable and that it coincides with $\pi_{1}\left(\left(\boldsymbol{N},\boldsymbol{B}\right)\right)$
$\hat{\mathbb{P}}$-a.s. It remains to consider the second level and
we only discuss the case $i\in\left\{ 1,\dots,d\right\} $ and $j\in\left\{ d+1,\dots,d+e\right\} $
(the other cases follow either immediately or by a similar argument).
To avoid confusion about probability space on which the involved stochastic
integrals are constructed, we use the notation $d_{\hat{\mathbb{P}}}$
resp.~$d_{\mathbb{P}}$. By definition of the Stratonovich lift,
\[
\left(\boldsymbol{N},\boldsymbol{B}\right){}_{.}^{(2);i,j}=\int_{0}^{.}N_{r}^{i}\circ d_{\hat{\mathbb{P}}}B_{r}^{j-d},\,\hat{\mathbb{P}}-a.s.
\]
and since by assumption the components of $N$ are independent of
$B$, above It\={o}-integral coincides with the It\={o}-version $\int_{0}^{.}N_{r}^{i}d_{\hat{\mathbb{P}}}B_{r}^{j-d}$.
By standard results 
\[
\left(\int_{0}^{t}N_{r}^{i}d_{\mathbb{\hat{\mathbb{P}}}}B_{r}^{j-d}\right)\left(\hat{\omega}\right)=\lim_{n\to\infty}\sum_{k=1}^{2^{n}}N_{(k-1)/2^{n}t}^{i}\left(\overline{\omega}\right)\left[B_{k/2^{n}t}^{j-d}\left(\omega\right)-B_{(k-1)/2^{n}t}^{j-d}\left(\omega\right)\right]\,\,\forall t
\]
holds for all $\hat{\omega}=\left(\overline{\omega},\omega\right)$
in some subset $\hat{A}\subset\hat{\mathcal{F}}$ of full measure,
$\hat{\mathbb{P}}\left[\hat{A}\right]=1$. By a Fubini type theorem
(e.g.~\cite[Theorem 3.4.1]{bibBogachev}), there exists a subset
$\overline{\Omega}^{\circ}\subset\overline{\Omega}$ of full measure,
such that for every $\overline{\omega}\in\overline{\Omega}^{\circ}$
the projection $\hat{A}_{\overline{\omega}}:=\left\{ \omega\in\Omega:\left(\overline{\omega},\omega\right)\in\hat{A}\right\} $
satisfies $\mathbb{P}\left[\hat{A}_{\overline{\omega}}\right]=1$.
On the other hand for every fixed $\overline{\omega}\in\overline{\Omega}$
\[
\boldsymbol{\Lambda}_{.}^{(2);i,j}\left(\boldsymbol{N}\left(\overline{\omega}\right)\right)=\int_{0}^{.}N_{r}^{i}\left(\overline{\omega}\right)d_{\mathbb{P}}B_{r}^{j-d}\,\,\,\mathbb{P}-a.s.
\]
and 
\[
\left(\int_{0}^{t}N_{r}^{i}\left(\overline{\omega}\right)d_{\mathbb{P}}B_{r}^{j-d}\right)\left(\omega\right)=\lim_{m\to\infty}\sum_{k=1}^{2^{n_{m}}}N_{(k-1)/2^{n_{m}}t}^{i}\left(\overline{\omega}\right)\left[B_{k/2^{n_{m}}t}^{j-d}\left(\omega\right)-B_{(k-1)/2^{n_{m}}t}^{j-d}\left(\omega\right)\right]
\]
for every $\omega\in D_{\overline{\omega}}$ where $D_{\overline{\omega}}\subset\Omega$
is a set of full measure, $\mathbb{P}\left[D_{\overline{\omega}}\right]=1$.
($D_{\overline{\omega}}$ as well as the subsequence $\left(n_{m}\right){}_{m}$
depends on $\overline{\omega}$). So for $\overline{\omega}\in\overline{\Omega}^{\circ},\omega\in\hat{A}_{\overline{\omega}}\cap D_{\overline{\omega}}$
we have that 
\begin{align*}
\left(\int_{0}^{t}N_{r}^{i}d_{\hat{\mathbb{P}}}B_{r}^{j-d}\right)\left(\overline{\omega},\omega\right) & =\lim_{n\to\infty}\sum_{k=1}^{2^{n}}N_{(k-1)/2^{n}t}^{i}\left(\overline{\omega}\right)\left[B_{k/2^{n}t}^{j-d}\left(\omega\right)-B_{(k-1)/2^{n}t}^{j-d}\left(\omega\right)\right]\\
 & =\lim_{m\to\infty}\sum_{k=1}^{2^{n_{m}}}N_{(k-1)/2^{n_{m}}t}^{i}\left(\overline{\omega}\right)\left[B_{k/2^{n_{m}}t}^{j-d}\left(\omega\right)-B_{(k-1)/2^{n_{m}}t}^{j-d}\left(\omega\right)\right]\\
 & =\left(\int_{0}^{t}N_{r}^{i}\left(\overline{\omega}\right)d_{\mathbb{P}}B_{r}^{j-d}\right)\left(\omega\right).
\end{align*}
The second equality holds since the sum converges along $n$, hence
also along any subsequence $\left(n_{m}\right){}_{m}$. Noting that
for every $\overline{\omega}\in\overline{\Omega}^{\circ}$ we have
$\mathbb{P}\left[\hat{A}_{\overline{\omega}}\cap D_{\overline{\omega}}\right]=1$
we can conclude that for $\overline{\mathbb{P}}$-a.e.~$\overline{\omega}\in\overline{\Omega}$
\begin{equation}
\left(\boldsymbol{N}\boldsymbol{,B}\right)_{.}=\boldsymbol{\Lambda}_{.}^{\boldsymbol{N}(\overline{\omega})}\,\,\,\mathbb{P}\text{-a.s.}\label{eqLiftEqualsLift-1}
\end{equation}
\end{proof}
\begin{rem}
The null-set on which the equality (\ref{eq:consistent}) holds depends
on $\boldsymbol{\eta}\in\mathcal{C}^{0,\alpha}\left(\mathbb{R}^{d+e}\right)$
(and the version of the stochastic process $B$), i.e.~the map $\boldsymbol{\eta}\mapsto\boldsymbol{\Lambda}^{\boldsymbol{\eta}}$
can be quite ``ugly'' from a measure-theoretic point of view. However,
Theorem \ref{thm:joint lift} shows that after taking expectations
(resp.~$L^{q}$ norms) this map is actually quite regular and we
will see that this is sufficient for important applications (Section
\ref{sec:differential-equations-with} and \ref{sec:Applications}).
\begin{rem}
In the construction we use the fact that the bracket between $N$
and $B$ is $0$. Especially, the consistency in Theorem \ref{thm:joint lift}
is only true for independent processes $N$ and $B$.
\begin{rem}
Lyons \cite{STMAZ.00008236} constructs a two-dimensional Gaussian
process such that its marginals are Brownian motions and shows that
for several different definitions of It\={o} and Stratonovich integrals
(as limit of Riemann sums, Fourier series approach) the cross-integrals
are only defined on a null-set. This does not contradict Theorem \ref{thm:joint lift}
due to the previous remark/the assumption of independence.%
\footnote{It is even not obvious if the process in \cite{STMAZ.00008236} has
a bracket.%
}
\begin{rem}
They key observation for the proof of Theorem \ref{thm:joint lift}
is that by assuming an integration by parts formula holds, the cross
integral can be implicitly defined. Especially, definition (\ref{eq:lambda2})
still makes sense if we replace Stratonovich by It\={o} integration
and one can run the above argument to arrive at a rough path lift
$\boldsymbol{\Lambda}^{Ito,\boldsymbol{\eta}}$ that is now a non-geometric
rough path (to be specific, one only needs to slightly change the
Fernique argument to account for the It\={o}--Stratonovich correction).
The proof of consistency follows also as above. Unfortunately, for
the application in non-linear filtering given in Section \ref{sec:Applications},
this does not lead to better results regarding the regularity of the
vector fields in the filtering problem.
\end{rem}
\end{rem}
\end{rem}
\end{rem}

\section{\label{sec:differential-equations-with}Rough and stochastic differential
equations (RSDEs)}

Our goal is to give meaning to the differential equation 
\[
dS_{t}^{\boldsymbol{\eta}}=a\left(S_{t}^{\boldsymbol{\eta}}\right)dt+b\left(S_{t}^{\boldsymbol{\eta}}\right)\circ dB_{t}+c\left(S_{t}^{\boldsymbol{\eta}}\right)d\boldsymbol{\eta}_{t},
\]
i.e.~for a fixed rough path $\boldsymbol{\eta}$ we want to find
a stochastic process $S^{\boldsymbol{\eta}}$ on the probability space
which carries the Brownian motion $B$. Theorem \ref{thm:joint lift}
guarantees the existence of a canonical, random joint lift $\boldsymbol{\Lambda}$
of $B$ and $\boldsymbol{\eta}$, hence we can solve for every fixed
rough path $\boldsymbol{\eta}\in\mathcal{C}_{d}^{0,\alpha}$ the random
RDE 
\begin{eqnarray*}
dS_{t} & = & a\left(S_{t}\right)dt+\left(b,c\right)\left(S_{t}\right)d\boldsymbol{\Lambda}_{t}\\
 & = & \left(a,b,c\right)\left(S_{t}\right)d\left(t,\boldsymbol{\Lambda}_{t}\right).
\end{eqnarray*}
Theorem \ref{thmRoughSDE} shows that this is indeed the right solution
in terms of consistent approximation results as well as continuity
of the solution map; Theorem \ref{thmRoughSDE'} shows consistency
with usual SDE solution in the case that $\boldsymbol{\eta}$ is the
rough path lift of another Brownian motion. Before we give the proofs
let us introduce some standard notation.
\begin{defn}
Let $\left(\Omega,\mathcal{F},\mathcal{F}_{t},\mathbb{P}\right)$
be a filtered probability space satisfying the usual condition. Denote
with $\mathcal{S}^{0}\left(\Omega\right)$ the space of adapted, continuous
processes in $\mathbb{R}^{d_{S}}$, with the topology of uniform convergence
in probability. For $q\ge1$ we denote with $\mathcal{S}^{q}\left(\Omega\right)$
the space of processes $X\in\mathcal{S}^{0}$ such that
\[
\left|X\right|{}_{\mathcal{S}^{q}}:=\left|\left|X\right|_{\infty;\left[0,t\right]}\right|_{L^{q}\left(\Omega;\mathbb{R}\right)}=\left(\mathbb{E}\left[\sup_{s\le t}\left|X_{s}\right|{}^{q}\right]\right)^{1/q}<\infty.
\]

\end{defn}

\subsection{Existence and continuity of the solution map}
\begin{assumption}
\label{assumption} $a\in Lip^{1+\epsilon}$ for some $\epsilon>0$,
$\alpha\in\left(\frac{1}{3},\frac{1}{2}\right)$, $\gamma>\frac{1}{\alpha}$
and $b,c\in Lip^{\gamma}$. 
\begin{thm}
\label{thmRoughSDE}Let $\left(\Omega,\mathcal{F},\mathcal{F}_{t},\mathbb{P}\right)$
be a filtered probability space satisfying the usual conditions, carrying
a $e$-dimensional Brownian motion $B$ and a random variable $S_{0}$
independent of $B$. Let $\left(a,b,c\right)$ fulfill Assumption
\ref{assumption} for some $\alpha\in\left(\frac{1}{3},\frac{1}{2}\right)$.
Then there exists a $d_{S}$-dimensional process $S^{\boldsymbol{\eta}}\in\mathcal{S}^{0}$
such that for every sequence $\left(\eta^{n}\right)_{n}$, $\eta^{n}\in C^{1}\left(\left[0,T\right],\mathbb{R}^{d}\right)$
and such that $\left(1+\eta^{n}+\int\eta^{n}\otimes d\eta^{n}\right)\rightarrow_{n}\boldsymbol{\eta}$
in $\rho_{\alpha-H\ddot{o}l}$-metric for some $\boldsymbol{\eta}\in\mathcal{C}^{0,\alpha}$,
the solutions $\left(S^{n}\right)_{n}$ of the SDE 
\begin{align*}
dS_{t}^{n} & =a\left(S_{t}^{n}\right)dt+b\left(S_{t}^{n}\right)\circ dB_{t}+c\left(S_{t}^{n}\right)d\eta_{t}^{n},\, S^{n}\left(0\right)=S_{0}^{n}.
\end{align*}
converge uniformly on compacts in probability to $S^{\boldsymbol{\eta}}$,
\begin{equation}
S^{n}\to_{n\rightarrow\infty}S^{\boldsymbol{\eta}}\text{ in }\mathcal{S}^{0}\label{eq:approximation solution}
\end{equation}
and the process $S^{\boldsymbol{\eta}}$ only depends on $\boldsymbol{\eta}$
and the process $B$ but not on the approximating sequence $\left(\eta^{n}\right)_{n}$.
We say that $S^{\boldsymbol{\eta}}$ is the solution of the RSDE 
\[
S_{t}^{\boldsymbol{\eta}}=S_{0}+\int_{0}^{t}a\left(S_{r}^{\boldsymbol{\eta}}\right)dr+\int_{0}^{t}b\left(S_{r}^{\boldsymbol{\eta}}\right)\circ dB_{r}+\int_{0}^{t}c\left(S_{r}^{\boldsymbol{\eta}}\right)d\boldsymbol{\eta}_{r}.
\]
Moreover, \end{thm}
\begin{enumerate}
\item $\forall q\ge1$, $S^{\boldsymbol{\eta}}\in\mathcal{S}^{q}\left(\Omega\right)$,
the map 
\begin{equation}
\left(\mathcal{C}^{0,\alpha},\rho_{\alpha-H\ddot{o}l}\right)\to\left(\mathcal{S}^{q}\left(\Omega\right),\left|.\right|{}_{\mathcal{S}^{q}}\right),\boldsymbol{\eta}\mapsto S^{\boldsymbol{\eta}}\label{eq:solution map}
\end{equation}
is locally Lipschitz continuous, 
\item If $S_{0}$ has Gaussian tails then $S$ also has Gaussian tails,
locally uniform in $\boldsymbol{\eta}$: $\forall r>0$ $\exists$$\delta=\delta\left(\alpha',\alpha,T,r\right)>0$
such that 
\[
\sup_{\left\Vert \boldsymbol{\eta}\right\Vert {}_{\alpha-H\ddot{o}l}\leq r}\mathbb{E}\left[\exp\left(\delta\left|S^{\boldsymbol{\eta}}\right|_{\infty;[0,T]}^{2}\right)\right]<\infty.
\]

\end{enumerate}
\end{assumption}
\begin{proof}
Choose $\alpha'<\alpha$ large enough, such that $\gamma>1/\alpha'$
and apply standard existence and uniqueness results%
\footnote{We only have $a\in Lip^{1+\epsilon}$ so we have to use results on
RDEs with drift (e.g.~\cite[Theorem 12.6 and Theorem 12.10 ]{friz-victoir-book})
to get existence of a unique solution. %
} to get a solution $S^{\boldsymbol{\eta}}$ of the RDE 
\begin{align}
S_{t}^{\boldsymbol{\eta}} & =S_{0}+\int_{0}^{t}a\left(S_{r}^{\boldsymbol{\boldsymbol{\eta}}}\right)dr+\int_{0}^{t}\left(b,c\right)\left(S_{r}^{\boldsymbol{\eta}}\right)d\boldsymbol{\Lambda}_{r}^{\boldsymbol{\eta}}.\label{eqJointRDE}
\end{align}
(and denote with $\boldsymbol{S}^{\eta}$ the full RDE solution).

\textbf{Point (1) (and (\ref{eq:approximation solution})). }Let $\alpha'<\alpha$
with $\gamma>p':=\frac{1}{\alpha'}$. By Theorem 4 in \cite{bibBayerFrizRiedel}
we have (see Section \ref{sec:integrability-and-tail} for the definition
of $N_{1}\left(\left\Vert \boldsymbol{S}^{\boldsymbol{\eta}}\right\Vert _{p^{\prime}-var}^{p^{\prime}};\left[0,T\right]\right)$)
\[
\left|S^{\boldsymbol{\eta}}-S^{\bar{\boldsymbol{\eta}}}\right|{}_{\infty}\le C\rho_{\alpha'-H\ddot{o}l}\left(S^{\boldsymbol{\eta}},S^{\bar{\boldsymbol{\eta}}}\right)\exp\left[c\left(N_{1}\left(\left\Vert \boldsymbol{S}^{\boldsymbol{\eta}}\right\Vert _{p^{\prime}-var}^{p^{\prime}};\left[0,T\right]\right)+N_{1}\left(\left\Vert \boldsymbol{S}^{\boldsymbol{\overline{\eta}}}\right\Vert _{p^{\prime}-var}^{p^{\prime}};\left[0,T\right]\right)+1\right)\right].
\]
Hence, using Theorem \ref{thm:joint lift} 
\begin{align*}
\left|S^{\boldsymbol{\eta}}-S^{\bar{\boldsymbol{\eta}}}\right|{}_{L^{q}} & \le C\left|\rho_{\alpha'-H\ddot{o}l}\left(S^{\boldsymbol{\eta}},S^{\bar{\boldsymbol{\eta}}}\right)\right|{}_{L^{2q}}\left|\exp\left[c\left(N_{1}\left(\left\Vert \boldsymbol{S}^{\boldsymbol{\eta}}\right\Vert _{p^{\prime}-var}^{p^{\prime}};\left[0,T\right]\right)+N_{1}\left(\left\Vert \boldsymbol{S}^{\boldsymbol{\overline{\eta}}}\right\Vert _{p^{\prime}-var}^{p^{\prime}};\left[0,T\right]\right)+1\right)\right]\right|{}_{L^{2q}}\\
 & \le C\rho_{\alpha-H\ddot{o}l}\left(\boldsymbol{\eta},\bar{\boldsymbol{\eta}}\right)\left|\exp\left[c\left(N_{1}\left(\left\Vert \boldsymbol{S}^{\boldsymbol{\eta}}\right\Vert _{p^{\prime}-var}^{p^{\prime}};\left[0,T\right]\right)+N_{1}\left(\left\Vert \boldsymbol{S}^{\boldsymbol{\overline{\eta}}}\right\Vert _{p^{\prime}-var}^{p^{\prime}};\left[0,T\right]\right)+1\right)\right]\right|{}_{L^{2q}}.
\end{align*}
The last $L^{2q}$-norm is bounded locally uniformly in $\boldsymbol{\eta},\overline{\boldsymbol{\eta}}$
by Corollary \ref{cor:integrabilityOfJointLift}. This yields the
desired local Lipschitzness of the solution map (\ref{eq:solution map}).
Now apply this continuity with the fact that if $\boldsymbol{\eta}$
is the lift of a smooth path $\eta$, then $S^{\boldsymbol{\eta}}$
is the standard SDE solution of the SDE 
\[
dS=a\left(S_{r}\right)dr+b\left(S_{r}\right)dB_{r}+c\left(S_{r}\right)d\eta_{r}.
\]
(e.g.~\cite[Section 17.5]{friz-victoir-book}).

\textbf{Point (2). }This follows from Corollary \ref{cor:integrabilityOfJointLift}
in combination with the pathwise estimates%
\footnote{We use the same notation as in \cite{friz-victoir-book}, $\left|S\right|_{0}\equiv\sup_{s\neq t}\left|S_{t}-S_{s}\right|$%
} 
\begin{align*}
\left|S^{\boldsymbol{\eta}}\right|{}_{\infty}\le\left|S_{0}^{\boldsymbol{\eta}}\right|+\left|S^{\boldsymbol{\eta}}\right|{}_{0} & \le\left|S_{0}^{\boldsymbol{\eta}}\right|+C\left|S^{\boldsymbol{\eta}}\right|{}_{p^{\prime}-\text{var}}\\
 & \le\left|S_{0}^{\boldsymbol{\eta}}\right|+C\left(N_{1}\left(\Lambda^{\boldsymbol{\eta}};\left[0,T\right]\right)+1\right),
\end{align*}
for some constant $C$. The last estimate follows from Lemma 4 and
Corollary 3 in \cite{bibFrizRiedel}.
\end{proof}

\subsection{Consistency with SDE solutions}
\begin{thm}
\label{thmRoughSDE'} Let $\left(\Omega,\mathcal{F},\mathcal{F}_{t},\mathbb{P}\right)$,
$B$,$S_{0}$ and $a,b,c$ be as in Theorem \ref{thmRoughSDE}. Let
$\left(\overline{\Omega},\overline{\mathcal{F}},\overline{\mathcal{F}}_{t},\overline{\mathbb{P}}\right)$
be another probability space satisfying the usual conditions and carrying
an $\overline{e}$-dimensional Brownian motion $\overline{B}$ and
denote 
\[
\left(\hat{\Omega},\hat{\mathcal{F}},\hat{\mathcal{F}}_{t},\hat{\mathbb{P}}\right)=\left(\Omega\times\overline{\Omega},\mathcal{F}\otimes\overline{\mathcal{F}}_{t},\mathcal{F}_{t}\otimes\overline{\mathcal{F}}_{t},\mathbb{P}\otimes\overline{\mathbb{P}}\right).
\]
Let $\hat{S}$ be the unique solution on $\left(\hat{\Omega},\hat{\mathcal{F}},\hat{\mathcal{F}}_{t},\hat{\mathbb{P}}\right)$
of the SDE 
\begin{align}
\hat{S}_{t} & =\hat{S}_{0}+\int_{0}^{t}a\left(\hat{S}_{r}\right)dr+\int_{0}^{t}b\left(\hat{S}_{r}\right)\circ dB_{r}+\int_{0}^{t}c\left(\hat{S}_{r}\right)\circ d\overline{B}_{r}.\label{eqStratonovich}
\end{align}
Denote with $\overline{\boldsymbol{B}}$ the Stratonovich lift of
the Brownian motion $\overline{B}$ on $\left(\overline{\Omega},\overline{\mathcal{F}},\overline{\mathcal{F}}_{t},\overline{\mathbb{P}}\right)$.
Then for $\mathbb{\overline{P}}$-a.e.~$\overline{\omega}\in\overline{\Omega}$
\begin{equation}
\mathbb{P}\left[\hat{S}_{t}\left(\overline{\omega},\cdot\right)=S_{t}^{\overline{\boldsymbol{B}}\left(\overline{\omega}\right)}\left(\cdot\right),t\in\left[0,T\right]\right]=1.\label{eqTheSame}
\end{equation}
\end{thm}
\begin{proof}
By Theorem \ref{thm:joint lift}, we know that for $\mathbb{\overline{P}}$-a.e.~$\overline{\omega}\in\overline{\Omega}$
we have 
\begin{equation}
\left(\boldsymbol{\overline{B},B}\right)_{.}=\boldsymbol{\Lambda}_{.}^{\overline{\boldsymbol{B}}\left(\overline{\omega}\right)}\,\,\,\mathbb{P}\text{-a.s.}\label{eqLiftEqualsLift}
\end{equation}
Standard results in rough path theory (cf.~\cite[Section 17.5]{friz-victoir-book}),
guarantee that the RDE solution to (\ref{eqJointRDE}) driven by $\left(\boldsymbol{\overline{B},B}\right)$
coincides $\hat{\mathbb{P}}$-a.s.~with the SDE solution of (\ref{eqStratonovich}).
Combining this with (\ref{eqLiftEqualsLift}) implies immediately
(\ref{eqTheSame}).
\end{proof}

\section{\label{sec:Applications}Applications}

In this section we show that RSDEs, as introduced in Section \ref{sec:differential-equations-with},
appear naturally in robustness questions of two important applications:
nonlinear filtering and stochastic/rough PDEs.

\subsection{Robustness in Nonlinear Filtering}

Nonlinear filtering is concerned with the estimation of a Markov process
based on some observation of it; e.g.~consider the classic case of
a Markov process $\left(X,Y\right)$ that takes values in $\mathbb{R}^{d_{X}+d_{Y}}$
of the form
\begin{equation}
\left\{ \begin{array}{rcll}
dX_{t} & = & l_{0}\left(X_{t},Y_{t}\right)dt+\sum_{k}Z\left(X_{t},Y_{t}\right)dB_{t}^{k}+\sum_{j}L_{j}\left(X_{t},Y_{t}\right)d\tilde{B}_{t}^{j} & \text{ (signal)}\\
dY_{t} & = & h\left(X_{t},Y_{t}\right)dt+d\tilde{B}_{t} & \text{ (observation)}
\end{array}\right.\label{eq:(x,y)}
\end{equation}
 with $B$ and $\tilde{B}$ independent, multidimensional Brownian
motions. The goal is to compute for a given real-valued function $\varphi$
\[
\pi_{t}\left(\varphi\right)=\mathbb{E}\left[\varphi\left(X_{t}\right)\lvert\sigma\left(Y_{r},r\in\left[0,t\right]\right)\right].
\]
From basic measure theory it follows that there exists a measurable
map 
\begin{equation}
\theta_{t}^{\varphi}:C\left(\left[0,T\right],\mathbb{R}^{d_{Y}}\right)\rightarrow\mathbb{R}\,\,\text{ such that }\theta_{t}^{\varphi}\left(Y\lvert_{\left[0,t\right]}\right)=\pi_{t}\left(\varphi\right)\,\,\,\mathbb{P}-a.s\label{eq:cond exp}
\end{equation}
In the late seventies Clark pointed out that this formulation is not
sufficient from a practical point of view: it would be natural to
demand that $\theta_{t}^{\varphi}\left(.\right)$ is continuous%
\footnote{The functional $\phi_{t}^{\varphi}$ is only uniquely defined up to
null-sets on pathspace but only discrete observations of $Y$ are
available. Moreoever, the model choosen for the observation process
might only be close in law to the ``real-world'' observation process.%
}. Clark showed that in the uncorrelated noise case (i.e.~$L\equiv0$
in \ref{eq:(x,y)}) there exists a unique 
\[
\overline{\theta}_{t}^{\varphi}:C\left(\left[0,T\right],\mathbb{R}^{d_{Y}}\right)\rightarrow\mathbb{R}
\]
which is continuous in uniform norm and fulfills (\ref{eq:cond exp}),
thus providing a ``robust version'' of the conditional expectation
$\pi_{t}\left(\varphi\right)$. Unfortunately, in the correlated noise
case this is no longer true (it is easy to construct counterexamples;
see \cite[Example 1]{CrisanDiehlFrizOberhauser}). Recently, it was
shown in \cite{CrisanDiehlFrizOberhauser} that also in this situation
robustness prevails, however only in a rough path sense, i.e.~there
exists a map 
\[
\overline{\theta}_{t}^{\varphi}:\mathcal{C}^{0,\alpha}\left(\mathbb{R}^{d_{Y}}\right)\rightarrow\mathbb{R}\text{ such that }\overline{\theta}_{t}^{\varphi}\left(\boldsymbol{Y}\lvert_{\left[0,t\right]}\right)=\pi_{t}\left(\varphi\right)\,\,\,\mathbb{P}-a.s
\]
here $\boldsymbol{Y}$ is the canonical rough path lift of the semimartingale
$Y$. The argument in \cite{CrisanDiehlFrizOberhauser} relies on
an observation of Mark Davis \cite{DavisRobustness}, namely that
under an appropriate change of measure the observation $Y$ is a Brownian
motion independent of $B$, the signal satisfies the SDE
\begin{equation}
dX_{t}=\overline{l}_{0}\left(X_{t},Y_{t}\right)dt+\sum_{k}Z_{k}\left(X_{t},Y_{t}\right)dY_{t}^{k}+\sum_{j}L_{j}\left(X_{t},Y_{t}\right)d\overline{B}_{t}^{j}\label{eq:X ito}
\end{equation}
where $\overline{l}_{0}=l_{0}+\sum_{k}Z_{k}h_{k}$ and that the robustness
question is linked to the (rough pathwise) robustness of $Y\mapsto X$.
To treat the resulting differential equation driven by Brownian noise
$B$ and a rough path (instead of $Y$) a flow decomposition is used
in \cite{CrisanDiehlFrizOberhauser}. We can now replace this argument
by Theorem \ref{thmRoughSDE} and Theorem \ref{thmRoughSDE'} which
leads to different regularity assumptions on the vector fields.
\begin{thm}
Let $\varphi\in Lip^{1}$ and let $\gamma>\frac{1}{\alpha}$ for some
\textup{$\alpha\in\left(\frac{1}{3},\frac{1}{2}\right)$} and 
\[
l_{0}\in Lip^{1+\epsilon},\,\,\, h,Z,L\in Lip^{\gamma}
\]
for some $\epsilon>0$. Denote with $\left(X,Y\right)$ be the solution
of (\ref{eq:(x,y)}). Then there exists a continuous map 
\[
\theta:\left(\mathcal{C}^{0,\alpha},\rho_{\alpha-H\ddot{o}l}\right)\rightarrow\left(\mathbb{R},\left|.\right|\right)
\]
such that 
\[
\theta\left(\boldsymbol{Y}\right)=\mathbb{E}\left[\varphi\left(X_{t}\right)\lvert\sigma\left(Y_{r},r\in\left[0,t\right]\right)\right]\,\,\,\mathbb{P}-a.s.
\]
where $\boldsymbol{Y}$ denotes the Stratonovich lift of the semimartingale
$Y$ to a geometric rough path.\end{thm}
\begin{proof}
To switch the equation (\ref{eq:X ito}) to the Stratonovich formulation
define 
\begin{eqnarray*}
L_{0}^{j}\left(x,y\right) & = & \overline{l}_{0}^{j}-\frac{1}{2}\sum_{k}\sum_{i}\partial_{x^{i}}Z_{k}^{j}\left(x,y\right)Z_{k}^{i}\left(x,y\right)-\frac{1}{2}\sum_{k}\partial_{y^{k}}Z_{k}^{j}\left(x,y\right)\\
 &  & -\frac{1}{2}\sum_{k}\sum_{i}\partial_{x^{i}}L_{k}^{j}\left(x,y\right)L_{k}^{i}\left(x,y\right)-\frac{1}{2}\sum_{k}\partial_{y^{k}}L_{k}^{j}\left(x,y\right).
\end{eqnarray*}
By Theorem \ref{thmRoughSDE}, 
\begin{eqnarray*}
dX_{t}^{\boldsymbol{\eta}} & = & L_{0}\left(X_{t}^{\boldsymbol{\eta}},Y_{t}^{\boldsymbol{\eta}}\right)dt+Z\left(X_{t}^{\boldsymbol{\eta}},Y_{t}^{\boldsymbol{\eta}}\right)d\boldsymbol{\eta}_{t}+\sum_{j}L_{j}\left(X_{t}^{\boldsymbol{\eta}},Y_{t}^{\boldsymbol{\eta}}\right)\circ d\overline{B}_{t}^{j}\\
dY_{t}^{\boldsymbol{\eta}} & = & d\boldsymbol{\eta}_{t}\\
dI_{t}^{\boldsymbol{\eta}} & = & h\left(X_{t}^{\boldsymbol{\eta}},Y_{t}^{\boldsymbol{\eta}}\right)d\boldsymbol{\eta}_{t}-\frac{1}{2}D_{k}h\left(X_{t}^{\boldsymbol{\eta}},Y_{t}^{\boldsymbol{\eta}}\right)dt
\end{eqnarray*}
has unique solution $\left(X_{t}^{\boldsymbol{\eta}},Y_{t}^{\boldsymbol{\eta}},I_{t}^{\boldsymbol{\eta}}\right)\in\mathcal{S}^{2}$
and following the proof of \cite[Theorem 6]{CrisanDiehlFrizOberhauser}
shows continuity of $\theta$ (for these steps it is important to
have $\mathbb{E}\left[\exp\left(qI_{t}^{\boldsymbol{\eta}}\right)\right]<\infty$
for $q\geq2$ as guaranteed by Theorem \ref{thmRoughSDE}). Similarly,
we can follow step-by-step \cite[Theorem 7]{CrisanDiehlFrizOberhauser}
to show the consistency $\theta\left(\boldsymbol{Y}\right)=\pi_{t}\left(\varphi\right)\,\,\,\mathbb{P}-a.s.$\end{proof}
\begin{rem}
The regularity assumption in \cite{CrisanDiehlFrizOberhauser} is
$h,Z\in Lip^{4+\epsilon}$,$L\in Lip^{1}$, i.e.~above approach allows
to relax the regularity of the sensor function $h$ and $Z$ by two
degrees of regularity for the price of an additional degree of regularity
of $L=\left(L_{i}\right)_{i=1}^{d}$.
\end{rem}

\subsection{Feynman--Kac representation for linear RPDEs}

Over the last years there has been an increased interest in giving
a (rough) pathwise meaning to stochastic partial differential equations
and several approaches have emerged, see for example Gubinelli et
al.~\cite{deya2011non,gubinelliLejayTindel}, Hairer et al\@. \cite{CPA:CPA20383,hairer2011solving}
and Teichmann \cite{teichmann2011another}. The approach we focus
on in this section is related to the work of Lions and Souganidis
\cite{lionsSouganidis} and Friz et al.~\cite{caruanaFriz,MR2765508,frizOberhauser,frizDiehl}.
In a setting similar to the one in \cite{frizOberhauser} we are able,
using rough SDEs, to prove existence and uniqueness of solutions under
weaker assumptions on the coefficients and give a stochastic representation
for the solution. 
\begin{defn}
\label{defRoughPDE}Let $\boldsymbol{\eta}\in\mathcal{C}^{0,\alpha}\left(\mathbb{R}^{d}\right)$
be a geometric rough path for some $\alpha\in\left(0,1\right]$ and
$\sigma:\mathbb{R}^{m}\rightarrow\mathbb{R}^{m'}$, $a:\mathbb{R}^{m}\rightarrow\mathbb{R}^{m}$,
$G_{i}:\mathbb{R}^{m}\times\mathbb{R}\times\mathbb{R}^{m}\rightarrow\mathbb{R}$
and $\phi:\mathbb{R}^{m}\rightarrow\mathbb{R}$ be such that for every
$\eta\in C^{1}\left(\left[0,T\right],\mathbb{R}^{d}\right)$ there
exists a unique bounded, uniformly continuous, viscosity solution
$v^{\eta}:\left[0,T\right]\times\mathbb{R}^{m}\rightarrow\mathbb{R}$
to
\[
\left\{ \begin{array}{rcl}
-dv^{\eta}-L\left(x,v^{\eta},Dv^{\eta},D^{2}v^{\eta}\right)dt-\sum_{i=1}^{d}G_{i}\left(x,v^{\eta},Dv^{\eta}\right)\dot{\eta}_{t}^{i} & = & 0,\\
v^{\eta}\left(T,x\right) & = & \phi\left(x\right),
\end{array}\right.
\]
 where

\[
L:\mathbb{R}^{m}\times\mathbb{R}^{m}\times\mathbb{S}^{m}\rightarrow\mathbb{R}\text{ is given as }L\left(x,p,M\right):=\operatorname{Tr}\left[\sigma\left(x\right)\sigma^{T}\left(x\right)M\right]+a\left(x\right)\cdot p.
\]
We then say that a bounded, uniformly continuous function $v:\left[0,T\right]\times\mathbb{R}^{m}\rightarrow\mathbb{R}$
is a solution of the \textit{\emph{rough partial differential equation
(RPDE)}}
\[
\left\{ \begin{array}{rcl}
-dv-L\left(x,v,Dv,D^{2}v\right)dt-c\left(x\right)\cdot Dvd\boldsymbol{\eta}_{t} & = & 0,\\
v\left(T,x\right) & = & \phi\left(x\right),
\end{array}\right.
\]
(with $c=\left(G_{1},\ldots,G_{d}\right)$), if for every sequence
of smooth paths $\left(\eta^{n}\right)_{n}\subset C^{1}\left(\left[0,T\right],\mathbb{R}^{d}\right)$
such that $\eta^{n}\to\boldsymbol{\eta}$ as $n\rightarrow\infty$
in rough path metric we have in locally uniform convergence 
\[
v^{\eta^{n}}\to v\text{ as }n\rightarrow\infty.
\]
\end{defn}
\begin{rem}
Of course above definition is only of use if one can show the existence
of a solution for an interesting family of $\left(L,c,\phi\right)$
in the above sense (uniqueness is built into the definition by the
uniqueness of the approximating solutions). This is still an area
of active research but for example if $c$ is affine linear there
exists a solution in above sense (see \cite{MR1799099,MR2765508,frizOberhauser}).
The theorem below shows not only the existence of such a solution
by a short proof relying on RSDEs as introduced in Section \ref{sec:differential-equations-with}
but gives additionally a Feynman--Kac representation. This finally
leads to lower regularity assumptions on the noise vector fields (however,
in contrast to \cite{MR1799099,MR2765508,frizOberhauser} it only
applies to linear operators $L$). \end{rem}
\begin{thm}
\label{thm:Let--be}Let $\boldsymbol{\eta}\in\mathcal{C}^{0,\alpha}$
be a geometric rough path, $\alpha\in\left(0,1\right]$. Assume $\gamma>\frac{1}{\alpha}$,
$\sigma,c\in Lip^{\gamma}$, $\zeta>1$, $a\in Lip^{\zeta}$. Assume
$\phi$ is bounded and uniformly continuous. Then, there exists a
unique solution to the RPDE%
\footnote{We use $-$ signs to emphasize that we treat a backward equation.%
}

\[
\left\{ \begin{array}{rcl}
-dv-L\left(x,v,Dv,D^{2}v\right)dt-c\left(x\right)\cdot Dvd\boldsymbol{\eta}_{t} & = & 0,\\
v\left(T,x\right) & = & \phi\left(x\right),
\end{array}\right.
\]
Moreover $v\left(t,x\right)=\mathbb{E}\left[\phi\left(S_{T}^{t,x}\right)\right]$
where $S^{s,x}$ denotes the solution of the RSDE
\begin{equation}
\left\{ \begin{array}{rcl}
dS_{t}^{s,x} & = & \overline{a}\left(S_{t}^{s,x}\right)dt+\sigma\left(S_{t}^{s,x}\right)\circ dB_{t}+c\left(S_{t}^{s,x}\right)d\boldsymbol{\eta}_{t},\\
S_{s}^{s,x} & = & x.
\end{array}\right.\label{eq:sde_fk}
\end{equation}
where ($\sigma^{i}$ denotes the $i$th column of $\sigma$) 
\[
\overline{a}=a-\frac{1}{2}\sum_{i=1}D_{\sigma^{i}}\sigma^{i}.
\]
\end{thm}
\begin{proof}
Let $\left(\eta^{n}\right)_{n}$ be a sequence of smooths paths converging
to $\boldsymbol{\eta}$ in rough path topology. For every fixed $n$
we have the Feynman--Kac representation (see e.g.~\cite[Theorem 4.13]{bibPardouxPeng92})
\begin{align*}
v^{n}\left(t,x\right) & =\mathbb{E}\left[\phi\left(S_{T}^{n,t,x}\right)\right],
\end{align*}
where $v^{n}$ is the unique, bounded viscosity solution to 
\begin{align*}
-dv^{n}-L\left(x,v^{n},Dv^{n},D^{2}v^{n}\right)dt-c\left(x\right)\cdot Dv^{n}d\eta_{t}^{n} & =0,\\
v^{n}\left(T,x\right) & =\phi\left(x\right),
\end{align*}
and $S^{n,s,x}$ solves the SDE 
\[
\left\{ \begin{array}{rcl}
dS_{t}^{n,s,x} & = & \overline{a}\left(S_{t}^{n,s,x}\right)dt+\sigma\left(S_{t}^{n,s,x}\right)\circ dB_{t}+c\left(S_{t}^{n,s,x}\right)d\eta_{t}^{n},\\
S_{s}^{n,s,x} & = & x.
\end{array}\right.
\]
Theorem \ref{thmRoughSDE'} now gives the pointwise convergence 
\[
v^{n}\left(t,x\right)=\mathbb{E}\left[\phi\left(S_{T}^{n,t,x}\right)\right]\to_{n\to\infty}\mathbb{E}\left[\phi\left(S_{T}^{t,x}\right)\right]=:v\left(t,x\right).
\]
To get locally uniform convergence, it suffices to show local equicontinuity
of $\left(v^{n}\right)_{n}$ (by the Arzelà--Ascoli theorem). By the
same arguments as in Theorem 6 one sees that a rough SDE is also locally
uniformly continuous in the initial condition $S_{0}$, uniformly
over $\boldsymbol{\eta}$ in bounded sets. Moreover, it is straightforward
to show, that for every $q\ge1$ 
\[
\mathbb{E}\left[\left\Vert \boldsymbol{\Lambda}^{\boldsymbol{\eta}}\right\Vert _{p-var;\left[0,t\right]}^{q}\right]=o\left(1\right)\text{ as }t\rightarrow0,
\]
locally uniformly for $\boldsymbol{\eta}$. Putting the above together
yields the local equicontinuity of the $\left(v^{n}\right)_{n}$. \end{proof}
\begin{rem}
Theorem \ref{thm:Let--be} can be easily extended to cover equations
of the type 
\begin{align*}
-dv-L\left(t,x,v,Dv,D^{2}v\right)dt-c\left(t,x,u,Du\right)d\boldsymbol{\eta}_{t} & =0,\\
v\left(T,x\right) & =\phi\left(x\right),
\end{align*}
where $c$ is affine linear in $\left(u,Du\right)$ (as in \cite{frizOberhauser}).
For brevity we only treat the gradient case. 
\begin{rem}
In Theorem \ref{thm:Let--be} we only assume $c\in Lip^{\gamma}$
in contrast to $Lip^{\gamma+2}$ as in \cite{MR2765508,frizOberhauser}
where a flow decomposition is used.
\end{rem}
\end{rem}

\section{\label{sec:integrability-and-tail}integrability estimates for gaussian
rough differential equations revisited}

A classic result of X.~Fernique \cite{STMAZ.03327878} shows that
Gaussian probability measures on separable Banach spaces have Gaussian
tails in the Banach norm. If one considers as Banach space an abstract
Wiener space, this immediately implies Gauss tails of norms of Gaussian
processes which is of uttermost importance for many applications in
stochastic analysis. In rough path norms, iterated stochastic integrals
additionally appear and Fernique's theorem is no longer directly applicable.
Another issue is that the genuine rough-\emph{pathwise} estimates%
\footnote{\label{fn:The-solution-}The solution $dy=V\left(y\right)d\boldsymbol{x}$
is estimated $\left|y_{t}\right|\leq c.\exp\left(c\left\Vert \boldsymbol{x}\right\Vert _{p-var}^{p}\right)$
and this is known to be\textbf{\emph{ }}rough-\emph{pathwise }optimal,
see \cite{friz-oberhauser-2008b}. Applied with $\boldsymbol{x}=\boldsymbol{B}$
and $p>2$ and the Gaussian tail property of $\left\Vert \boldsymbol{B}\right\Vert _{p-var}$
this does not even imply the integrability of the RDE solution.%
} for solutions of RDEs driven by Gaussian processes do not ``see''
probabilistic cancellations, hence do not lead to useful probabilistic
estimates (e.g.~$L^{q}\left(\Omega\right)$ estimates) for solutions
of such RDEs. 

In \cite[Theorem 2]{MR2661566} the Borell--Sudakov--Tsirelson inequality
--- an analogue of the Gaussian isoperimetric inequality which holds
in infinite dimensional spaces --- was used to prove a generalization
of Fernique's theorem. This implies for example that $\left\Vert \boldsymbol{B}\right\Vert _{p-var}$
has Gauss tails for $p>2$ (see also our proof of Theorem \ref{thm:joint lift})
but combined only with pathwise estimates for RDE solutions this is
not even sufficient to derive moment estimates for RDE solutions driven
by Brownian motion (see footnote \ref{fn:The-solution-}; in It\={o}'s
stochastic calculus this is of course easy to establish). A key insight
was recently made in \cite{bibCassLyonsLitterer} by introducing ``\emph{greedy
partitions}'' which allow to capture the needed probabilistic cancellations.
The main result in \cite{bibCassLyonsLitterer} can then be seen as
the verification that a certain random measure $N$ (which is related
to the norm of a Gaussian rough path along such greedy partitions,
Definition \ref{defn:cll_random_measure}), has exponential tails
on compact sets (or even Gaussian tails in the case of Brownian motion).
The proof also uses the Borell--Sudakov--Tsirelson inequality. In
this section, using the isoperimetric inequality in a slightly different
spirit, we give another proof of the main result in \cite{bibCassLyonsLitterer}.
Our proof, based on a generalization of \cite[Theorem 2]{MR2661566}
and the greedy partitions of \cite{bibCassLyonsLitterer}, is surprisingly
short and, as we hope, may be somewhat more instructive.

\subsection{Revisiting the generalized Fernique theorem}

We first present a generalization of \cite[Theorem 2]{MR2661566}
which can be stated in a fairly general framework. Let $E$ be a real,
locally convex Hausdorff space. A measure $\gamma$ on the Borel sets
of $E$ is called a (centered) \textit{Gauß\ measure} if the push
forward measure under each element of the topological dual of $E$
is a (centered) normal random variable in $\mathbb{R}$. The corresponding
Cameron--Martin space will be denoted by $\mathcal{H}$. The triplet
$\left(E,\mathcal{H},\gamma\right)$ will be called a \textit{Gaussian
space}. $\gamma$ is called a \textit{Radon probability measure} on
the Borel sets of $E$ if $\gamma\left(B\right)=\gamma_{*}\left(B\right)$
for every Borel set $B$ where, for any subset $A\subset E$, 
\[
\gamma_{*}\left(A\right):=\sup\left\{ \gamma\left(K\right)\,:\, K\text{ compact and }K\subseteq A\right\} .
\]

\begin{thm}
\label{thm:even_more_gen_fernique} Let $\left(E,\mathcal{H},\gamma\right)$
be a Gaussian space with $\gamma$ being centered and a Radon measure%
\footnote{Note that probability measures on the Borel sets of Polish spaces
are Radon measures, thus Gaussian measures on separable Fréchet spaces
(and therefore on Banach spaces) are always Radon measures.%
}. Let $f,g\colon E\to\mathbb{R}\cup\left\{ +\infty,-\infty\right\} $
be measurable functions. Assume that there is a null-set $N$ such
that for every $x$ outside $N$ we have 
\[
\left|f\left(x\right)\right|\leq\left|g\left(x-h\right)\right|+\sigma\left|h\right|{}_{\mathcal{H}}
\]
for every $h\in\mathcal{H}$. Assume further that there is an $r_{0}\geq0$
such that 
\[
\gamma\left\{ x\in E\,:\,\left|g\left(x\right)\right|\leq\frac{r_{0}}{2}\right\} =:a>0.
\]
Then 
\[
\gamma\left\{ x\in E\,:\,\left|f\left(x\right)\right|>r\right\} \leq1-\Phi\left(\alpha+\frac{r}{2\sigma}\right)
\]
for every $r\geq r_{0}$ where $\Phi$ denotes the cumulative distribution
function of a standard normal random variable and $\alpha\in\mathbb{R}$
is chosen such that $\Phi\left(\alpha\right)\leq a$. \end{thm}
\begin{proof}
Inspection of the proof in \cite{MR2661566} shows that the very same
argument holds when $f\left(x-h\right)$ is replaced by $g\left(x-h\right)$. 
\end{proof}

\subsection{Greedy partitions}

Recall the following definition from \cite{bibCassLyonsLitterer}.
\begin{defn}
\label{defn:cll_random_measure}Let $\omega\colon\left\{ u,v\in\left[0,T\right]{}^{2}\,:\, u\leq v\right\} \to\mathbb{R}_{+}$
be a control function \cite{friz-victoir-book}. Let $\left[s,t\right]\subseteq\left[0,T\right]$
and choose $\beta>0$. Define $\left\{ \tau_{0}\leq\tau_{1}\leq\ldots\right\} $
as 
\begin{align*}
\tau_{0} & =s\\
\tau_{i+1} & =\inf\left\{ u\,:\,\omega\left(\tau_{i},u\right)\geq\beta,\ \tau_{i}<u\leq t\right\} \wedge t.
\end{align*}
Then we set $N_{\beta}\left(\omega;\left[s,t\right]\right):=\sup\left\{ n\in\mathbb{N}_{0}\,:\,\tau_{n}<t\right\} $.
If $\mathbf{x}\colon\left[0,T\right]\to G^{[p]}\left(\mathbb{R}^{d}\right)$
is a weakly geometric $p$-rough path and $\|\cdot\|_{p-\text{var}}$
denotes the homogeneous $p$-variation norm induced by the Carnot--Caratheodory
norm (cf.~\cite[Chapter 8]{friz-victoir-book}), we set $N_{\beta}\left(\mathbf{x};\left[s,t\right]\right):=N_{\beta}\left(\|\mathbf{x}\|_{p-\text{var}}^{p};\left[s,t\right]\right)$.\end{defn}
\begin{lem}
\label{lemma:growth_random_meas_cm_shift}Let $\mathbf{x}$ be a weakly
geometric $p$-rough path and $h$ be a path of bounded $q$-variation
where $1\leq q\leq p$ and $\frac{1}{p}+\frac{1}{q}>1$. Then there
is an $\beta=\beta\left(p,q\right)$ such that%
\footnote{$T_{h}$ denotes the usual translation operator, see \cite[Chapter 9]{friz-victoir-book}.%
} 
\[
N_{\beta}\left(T_{h}\left(\mathbf{x}\right);\left[0,T\right]\right)\leq\|\mathbf{x}\|_{p-\text{var}}^{p}+\left|h\right|{}_{q-\text{var}}^{q}.
\]
\end{lem}
\begin{proof}
We have 
\begin{align*}
N_{\beta}\left(\|T_{h}\left(\mathbf{x}\right)\|_{p-\text{var}}^{p};\left[0,T\right]\right) & \leq N_{\beta}\left(C_{p,q}\left(\|\mathbf{x}\|_{p-\text{var}}^{p}+\|h\|_{q-\text{var}}^{p}\right);\left[0,T\right]\right)\\
 & =N_{1}\left(\|\mathbf{x}\|_{p-\text{var}}^{p}+\|h\|_{q-\text{var}}^{p};\left[0,T\right]\right)
\end{align*}
with the choice $\beta=C_{p,q}$, using \cite[Theorem 9.33]{friz-victoir-book}.
By definition, 
\[
N_{1}\left(\|\mathbf{x}\|_{p-\text{var}}^{p}+\|h\|_{q-\text{var}}^{p};\left[0,T\right]\right)\leq\sum_{\tau_{i}}\|\mathbf{x}\|_{p-\text{var};\left[\tau_{i},\tau_{i+1}\right]}^{p}+\|h\|_{q-\text{var};\left[\tau_{i},\tau_{i+1}\right]}^{p}
\]
where $\left(\tau_{i}\right)$ is a finite partition of $\left[0,T\right]$
for which $\|\mathbf{x}\|_{p-\text{var};\left[\tau_{i},\tau_{i+1}\right]}^{p}+\|h\|_{q-\text{var};\left[\tau_{i},\tau_{i+1}\right]}^{p}\leq1$
for every $\tau_{i}$, and in particular $\|h\|_{q-\text{var};\left[\tau_{i},\tau_{i+1}\right]}^{p}\leq\|h\|_{q-\text{var};\left[\tau_{i},\tau_{i+1}\right]}^{q}$.
Hence 
\begin{eqnarray*}
N_{1}\left(\|\mathbf{x}\|_{p-\text{var}}^{p}+\|h\|_{q-\text{var}}^{p};\left[0,T\right]\right) & \leq & \sum_{\tau_{i}}\|\mathbf{x}\|_{p-\text{var};\left[\tau_{i},\tau_{i+1}\right]}^{p}+\|h\|_{q-\text{var};\left[\tau_{i},\tau_{i+1}\right]}^{q}\\
 & \leq & \|\mathbf{x}\|_{p-\text{var};\left[0,T\right]}^{p}+\|h\|_{q-\text{var};\left[0,T\right]}^{q}.
\end{eqnarray*}

\end{proof}

\subsection{Integrability estimates for rough path valued random variables}

Combining the above leads to a simple and easy proof of integrability
estimates for Gaussian rough path norms.
\begin{thm}[Integrability of rough path valued random variables]
 \label{thm:integrabilityOfRVs} Let $\left(\Omega,\mathcal{H},\gamma\right)$
be a centered Gaussian space with $\Omega=C_{0}\left(\left[0,T\right],\mathbb{R}^{d}\right)$.
Assume that there is a measurable map $F\colon\Omega\to\mathcal{C}^{p}$
to the space of geometric $p$-rough paths. Furthermore, assume that
there is an embedding 
\begin{align}
\iota\colon\mathcal{H}\hookrightarrow C^{q-\text{var}}\label{eqn:CM_embedding}
\end{align}
with $1\leq q\leq p$ and $\frac{1}{p}+\frac{1}{q}>1$ and that the
set 
\begin{align}
\left\{ \omega\,:\, T_{h}\left(F\left(\omega\right)\right)=F\left(\omega+h\right)\text{ for all }h\in\mathcal{H}\right\} =:\tilde{\Omega}\label{eq:shift}
\end{align}
has full measure. Then for all $\beta>0$, $N_{\beta}\left(F;\left[0,T\right]\right)^{\frac{1}{q}}$
has Gaussian tails. More specific, if 
\[
\mathbb{P}\left[\left\Vert F\right\Vert {}_{p-var}\leq K\right]\geq a>0
\]
and if $M$ is a bound on $\|\iota\|_{\mathcal{H}\hookrightarrow C^{q-\text{var}}}$,
there is a $\delta=\delta\left(p,q,K,a,M,\beta\right)>0$ such that
\[
\mathbb{E}\left[\exp\left(\delta N_{\beta}\left(F;\left[0,T\right]\right){}^{\frac{2}{q}}\right)\right]<\frac{1}{\delta}.
\]
 \end{thm}
\begin{proof}
Lemma \ref{lemma:growth_random_meas_cm_shift} implies that there
is a $\beta_{0}$ such that 
\[
N_{\beta_{0}}\left(F\left(\omega\right);\left[0,T\right]\right)\leq\|F\left(\omega-h\right)\|_{p-\text{var}}^{p}+\|\iota\|_{\mathcal{H}\hookrightarrow C^{q}}\left|h\right|{}_{\mathcal{H}}^{q}
\]
holds on the set $\tilde{\Omega}$ for every $h\in\mathcal{H}$. Thus
we may apply Theorem \ref{thm:even_more_gen_fernique} to conclude
the assertion for $\beta_{0}$. By Lemma 3 in \cite{bibFrizRiedel},
$N_{\beta}$ and $N_{\beta'}$ are comparable for all $\beta,\beta'>0$.
We hence get the stated result for all $\beta>0$. 
\end{proof}
If the covariance of a Gaussian process has finite $\rho$-variation
for some $\rho<2$, it can be lifted in the sense of Friz--Victoir,
cf. \cite{friz-victoir-2007-gauss}. Finite $\rho$-variation of the
covariance also implies the embedding (\ref{eqn:CM_embedding}) with
$q=\rho$, cf. \cite[Proposition 17]{friz-victoir-2007-gauss}, which
means that \ref{eqn:CM_embedding} is fulfilled whenever $\rho<3/2$.
A slightly stronger condition, so called mixed $\left(1,\rho\right)$-variation,
was seen to imply an even sharper embedding with $q=\frac{2}{\rho^{-1}+1}$,
cf.~\cite{2013arXiv1307.3460F}, thus condition \ref{eqn:CM_embedding}
holds for all $\rho<2$. Choosing $q$ according to one of these embeddings,
we obtain 
\begin{cor}[Integrability of Gaussian rough paths]
Let $\left(\Omega,\mathcal{H},\gamma\right)$ be a centered Gaussian
space with $\Omega=C_{0}\left(\left[0,T\right],\mathbb{R}^{d}\right)$
and and let $X\colon\Omega\to C_{0}\left(\left[0,T\right],\mathbb{R}^{d}\right)$
denote the coordinate process. Assume that all components of $X$
are independent and that the 2--dimensional $\rho$-variation of the
covariance function $R$ of $X$ is finite for some $\rho<2$. Let
$\mathbf{X}$ denote the lift of $X$ in the sense of Friz--Victoir.
Then $N_{\beta}\left(\mathbf{X};\left[0,T\right]\right){}^{1/q}$
has Gaussian tails for every $\beta>0$. \end{cor}
\begin{proof}
 By construction of the lift, $\mathbf{X}$ takes values in $\mathcal{C}^{0,p}$
almost surely, and (\ref{eq:shift}) holds by \cite[Proposition 15.58]{friz-victoir-book}.
We thus conclude with Theorem \ref{thm:integrabilityOfRVs}. 
\end{proof}

\subsection{Application: Integrability of RSDE solutions}

We now apply these general results to Rough and Stochastic differential
equations (RSDEs) as introduced in Section \ref{sec:differential-equations-with}.
First we need a Lemma.
\begin{lem}
\label{lem:translation},\label{lemma:suff_crit_gaussian_int_joint_lift}~

i) For the joint lift $\boldsymbol{\Lambda}^{\boldsymbol{\eta}}$
from Theorem \ref{thm:joint lift} we have 
\[
\mathbb{P}\left[\boldsymbol{\Lambda}^{\boldsymbol{\eta}}\left(\omega+h\right)=T_{h}\boldsymbol{\Lambda}^{\boldsymbol{\eta}}\left(\omega\right)\ \forall h\in\mathcal{H}\left(\mathbb{R}^{e}\right)\right]=1
\]

ii) For all $r>0$ and $p>\frac{1}{\alpha}$ there is a $k$ such
that 
\[
\inf_{\|\mathbf{\boldsymbol{\eta}}\|_{\alpha-H\ddot{o}l}<r}\mathbb{P}\left[\|\boldsymbol{\Lambda}^{\boldsymbol{\eta}}\|_{p-var}\leq k\right]\ge\frac{1}{2}.
\]
\end{lem}
\begin{proof}
i) Let $D=\left\{ 0=t_{0}<\ldots<t_{m}=T\right\} $ be any partition
of $\left[0,T\right]$, $\left|D\right|$ denotes its mesh size. Let
$B^{D}$ be the piecewise linear approximation of $B$ on the partition
$D$. An easy calculation shows that 
\[
\int\eta_{0,r}dB_{r}^{D}=\sum_{i}\eta_{0,\bar{t}_{i}}B_{t_{i},t_{i+1}}
\]
for some deterministic $\bar{t}_{i}\in[t_{i},t_{i+1}]$ and where
the integral on the left hand side is defined as Riemann--Stieltjes
integral. As a consequence, 
\begin{align*}
\left|\int\eta_{0,r}dB_{r}-\int\eta_{0,r}dB_{r}^{D}\right|_{L^{2}} & \le\left|\int\eta_{0,r}dB_{r}-\sum_{i}\eta_{0,t_{i}}B_{t_{i},t_{i+1}}\right|_{L^{2}}\\
 & \qquad+\left|\sum_{i}\eta_{0,t_{i}}B_{t_{i},t_{i+1}}-\sum_{i}\eta_{0,\bar{t}_{i}}B_{t_{i},t_{i+1}}\right|_{L^{2}}.
\end{align*}
Now the first term converges to zero, as $|D|\to0$, by definition
of the It\={o} integral as limit of left-point Riemann sums. Using
the fact that $\eta$ is deterministic, we dominated the second term
by 
\[
\left(\sum_{i}\left|\eta_{t_{i}}-\eta_{\bar{t}_{i}}\right|\left|t_{i+1}-t_{i}\right|\right)^{1/2},
\]
which converges to $0$ as $\left|D\right|\to0$, by continuity of
$\eta$.

Using this characterization of the It\={o} integral as the limit of
smooth integrals, we can now finish the proof using exactly the same
argument as in \cite[Proposition 15.58]{friz-victoir-book}.

ii) If $\|\mathbf{\boldsymbol{\eta}}\|_{\alpha-H\ddot{o}l}<r$, by
Markov's inequality and Theorem \ref{thm:joint lift}, 
\[
\mathbb{P}\left[\|\boldsymbol{\Lambda}^{\boldsymbol{\eta}}\|_{p-var}\leq k\right]\geq1-\frac{C}{k}
\]
where $C$ is a constant depending on $r$. Choosing $k$ large enough
gives the result. \end{proof}
\begin{cor}[Integrability of joint lift]
 \label{cor:integrabilityOfJointLift} Let $\boldsymbol{\Lambda}^{\boldsymbol{\eta}}$
be the joint lift from Theorem \ref{thm:joint lift} with sample paths
in a $p$-rough paths space with $p>\frac{1}{\alpha}$. Then $N_{\beta}\left(\boldsymbol{\Lambda}^{\boldsymbol{\eta}};\left[0,T\right]\right)$
has Gaussian tails for every $\beta>0$. More specific, for every
$r>0$ there is a $\delta=\delta\left(p,\alpha,\beta,r\right)>0$
such that 
\[
\sup_{\eta\,:\,\|\boldsymbol{\eta}\|_{\alpha-H\ddot{o}l}\leq r}\mathbb{E}\left[\exp\left(\delta N_{\beta}\left(\boldsymbol{\Lambda}^{\boldsymbol{\eta}};\left[0,T\right]\right)^{2}\right)\right]\leq\frac{1}{\delta}.
\]
\end{cor}
\begin{proof}
For the Brownian motion, (\ref{eqn:CM_embedding}) holds with $q=1$
and $\|\iota\|_{\mathcal{H}\hookrightarrow C^{1-\text{var}}}\leq\sqrt{T}$,
cf. \cite[Proposition 17]{friz-victoir-2007-gauss}. The assertion
follows from Theorem \ref{thm:integrabilityOfRVs} and Lemma \ref{lemma:suff_crit_gaussian_int_joint_lift}.\bibliographystyle{plain}
\bibliography{/home/hd/Dropbox/projects/BibteX/roughpaths}
\end{proof}

\end{document}